\documentclass[letterpaper, 10 pt, conference]{ieeeconf}

\IEEEoverridecommandlockouts

\makeatletter
\newcommand{\removelatexerror}{\let\@latex@error\@gobble}
\makeatother
\usepackage[T1]{fontenc}
\usepackage[utf8]{inputenc}

\usepackage[english]{babel}
\usepackage[usenames,dvipsnames,table]{xcolor}
\usepackage{amsmath,amsthm,amssymb}
\usepackage{graphicx}
\usepackage{url}
\usepackage{float}

\usepackage{tikz}
\usetikzlibrary{shapes,arrows,positioning,automata}

\usepackage[mathscr]{euscript}
\usepackage{mathtools}
\DeclarePairedDelimiter{\ceil}{\lceil}{\rceil}
\usepackage{xargs} 
\usepackage{subfigure}
\usepackage[font=footnotesize,labelfont=footnotesize]{caption}
\usepackage[ruled,shortend,commentsnumbered,linesnumbered]{algorithm2e}

\usepackage{etoolbox} \let\bbordermatrix\bordermatrix
\patchcmd{\bbordermatrix}{8.75}{4.75}{}{}
\patchcmd{\bbordermatrix}{\left(}{\left[}{}{}
    \patchcmd{\bbordermatrix}{\right)}{\right]}{}{}

\newcommand{\real}{\mathbb{R}}
\newcommand{\realnonnegative}{{\mathbb{R}}_{\ge 0}}
\newcommand{\realpositive}{\mathbb{R}_{>0}}
\newcommand{\naturalnumbers}{\mathbb{N}}
\newcommand{\norm}[1]{\ensuremath{\| #1 \|}}
\newcommand{\until}[1]{[#1]}
\newcommand{\map}[3]{#1:#2 \rightarrow #3}
\newcommand{\setdef}[2]{\{#1 \; | \; #2\}}
\newcommand{\argmin}{\operatorname{argmin}}
\newcommand{\myemphc}[1]{\emph{#1}} 
 
\newcommand{\UUhat}{\widehat{\UU}}
\newcommand{\uhat}{\widehat{u}}
\newcommand{\abs}[1]{|#1|}
\newcommand{\absb}[1]{\Bigl|#1\Bigr|}

\newcommand{\Eb}{\mathbb{E}}
\newcommand{\Pb}{\mathbb{P}}

\renewcommand{\AA}{\mathcal{A}}
\newcommand{\CC}{\mathcal{C}}
\newcommand{\DD}{\mathcal{D}}
\newcommand{\EE}{\mathcal{E}}
\newcommand{\FF}{\mathcal{F}}
\newcommand{\GG}{\mathcal{G}}
\newcommand{\HH}{\mathcal{H}}

\newcommand{\OO}{\mathcal{O}}

\newcommand{\PP}{\mathcal{P}}
\renewcommand{\SS}{\mathcal{S}}
\newcommand{\TT}{\mathcal{T}}
\newcommand{\UU}{\mathcal{U}}
\newcommand{\VV}{\mathcal{V}}
\newcommand{\WW}{\mathcal{W}}
\newcommand{\XX}{\mathcal{X}}
\newcommand{\YY}{\mathcal{Y}}

\newcommand{\CVaR}{\operatorname{CVaR}}
\newcommand{\VaR}{\operatorname{VaR}}
\newcommand{\CVaRhat}{\widehat{\CVaR}}
\newcommand{\Fhat}{\widehat{F}}
 
\newcommand{\SSCWE}{\mathcal{S}_{\mathtt{CWE}}}
\newcommand{\SSCWEhat}{\widehat{\mathcal{S}}_{\mathtt{CWE}}}
\newcommand{\Db}{\mathbb{D}}

\newcommand{\hhat}{\widehat{h}}

\newcommand{\psihat}{\widehat{\psi}}

\newcommand{\dist}{\operatorname{dist}}
\newcommand{\eps}{\epsilon}
\newcommand{\teehat}{\widehat{t}}
\newcommand{\tee}{t}

\newcommand{\VI}{\operatorname{VI}}
\newcommand{\SOL}{\operatorname{SOL}}
\newcommand{\diam}{\operatorname{diam}}
\newcommand{\vol}{\operatorname{vol}}
\newcommand{\htil}{\tilde{h}}

\newcommand{\tb}{\bar{t}}

\newcommand{\oprocendsymbol}{\hbox{$\bullet$}}
\newcommand{\oprocend}{\relax\ifmmode\else\unskip\hfill\fi\oprocendsymbol}
\newcommand{\longthmtitle}[1]{\mbox{}\textup{\textsl{(#1):}}}

\allowdisplaybreaks

\newcommand{\ifinclude}[1]{}
\renewcommand{\ifinclude}[1]{#1}

\renewcommand{\theenumi}{(\roman{enumi})}

\makeatletter
\newcommand{\thickhline}{%
  \noalign {\ifnum 0=`}\fi \hrule height 1pt
  \futurelet \reserved@a \@xhline
}
\newcolumntype{"}{@{\hskip\tabcolsep\vrule width 1pt\hskip\tabcolsep}}
\makeatother

\newtheorem{theorem}{Theorem}[section]
\newtheorem{proposition}[theorem]{Proposition}
\newtheorem{lemma}[theorem]{Lemma}

\theoremstyle{definition}
\newtheorem{definition}[theorem]{Definition}
\newtheorem{assumption}[theorem]{Assumption}
\newtheorem{remark}[theorem]{Remark}

\usepackage[normalem]{ulem} 
\definecolor{new}{rgb}{0.55,0,0.55}
\usepackage[prependcaption,colorinlistoftodos]{todonotes}

\title{Sample average approximation of CVaR-based Wardrop equilibrium \\ in routing under uncertain costs} 
\author{Ashish Cherukuri \thanks{The author is with the Engineering and Technology Institute Groningen, University of Groningen. Email: \texttt{a.k.cherukuri@rug.nl.}}}

\begin{document}

\maketitle
\thispagestyle{empty}
\pagestyle{empty}

\begin{abstract}
	This paper focuses on the class of routing games that have uncertain costs. Assuming that agents are risk-averse and select paths with minimum conditional value-at-risk (CVaR) associated to them, we define the notion of CVaR-based Wardrop equilibrium (CWE). We focus on computing this equilibrium under the condition that the distribution of the uncertainty is unknown and a set of independent and identically distributed samples is available. To this end, we define the sample average approximation scheme where CWE is estimated with solutions of a variational inequality problem involving sample average approximations of the CVaR.  We establish two properties for this scheme. First, under continuity of costs and boundedness of uncertainty, we prove asymptotic consistency, establishing almost sure convergence of approximate equilibria to CWE as the sample size grows. Second, under the additional assumption of Lipschitz cost, we prove exponential convergence where the probability of the distance between an approximate solution and the CWE being smaller than any constant approaches unity exponentially fast. Simulation example validates our theoretical findings. 
\end{abstract}

\section{Introduction}\label{sec:intro}
Users of a transportation network are often selfish, minimizing their own cost function, such as travel time, when traversing through the network. This phenomenon is popularly modeled as a (deterministic) nonatomic routing game where the number of users is assumed to be large, each controlling infinitesimal amount of flow in the network. Therefore, an individual user does not affect the cost incurred on a path by unilaterally changing its route choice. At an equilibrium of this game, termed Wardrop equilibrium (WE)~\cite{wardrop1952some}, paths with nonzero flow have the least cost among all the alternatives. In real-life, the cost associated to a path is uncertain, affected by unplanned events such as, accidents, weather fluctuations, and construction work. Different users might minimize different objectives under this uncertainty, e.g., expected cost, specific quantile of the cost, or a risk measure. These disparate behaviors lead to different notions of equilibrium. Computing these equilibria and analyzing their properties help predict congestion patterns. Motivated by this, our goal here is to estimate the equilibria, using samples of the uncertainty, when agents are risk-averse and seek paths that have minimum conditional value-at-risk. 

\subsubsection*{Literature review}

In a routing setup, experimental studies~\cite{monnot2017routing,lam2016learning} validate the fact that agents arrive at some equilibrium path choice after repeated interaction with each other. Traditionally, equilibrium flow was predicted in a deterministic setup under the formalism of WE~\cite{correa2011wardrop}. 
Such predictions were used for designing tolls~\cite{brown2017studies} and planning for future transportation infrastructure~\cite{florian1999untangling}. When costs are uncertain, no single traffic flow works as a WE for all realizations of the uncertainty.  
Therefore, to estimate equilibrium flow, some works solve a stochastic nonlinear complementarity problem \cite{shanbhag2013tutorial}, either in an expectation basis \cite{zhang2011robust,XC-RJBW-YZ:12} or in a robust manner \cite{ordonez2010wardrop,xie2016robust}. In~\cite{JD-ARH-AC:19-ecc}, equilibrium flow is hypothesized to be the minimizer of the regret experienced by users. These approaches might be too conservative or might not account for the risk-sensitive behavior of agents. 
Among the works that consider risk,~\cite{nikolova2014mean-risk} and~\cite{prakash2018risk-averse} consider the cost to be the weighted sum of the mean and the variance of the uncertain cost. Further, for this risk criteria,~\cite{nikolova2015burden} introduces the notion of price of risk aversion and~\cite{nikolova2018informs} determines tighter bounds for it. 
In the transportation literature, the CVaR-based equilibrium is also known as the mean excess traffic equilibrium, see e.g.,~\cite{chen2010alpha-reliable,xu2017lmete} and references therein. While these works have explored numerous algorithms for computing the equilibrium, they lack theoretical performance guarantees for sample-based solutions. This paper attempts to bridge this gap. 

On the technical side, our paper relates to the body of work on sample average approximation, see~\cite[Chapter 5]{shapiro2014lectures} for a detailed overview. In particular, we borrow ideas from studies on sample average approximation of stochastic variational inequalities~\cite{xu2010saastochasticvi}, generalized equations~\cite{xu2010uniformexp}, and mathematical programs with equilibrium constraints~\cite{shapiro2008saa-empec}.

\subsubsection*{Setup and contributions} 
Our starting point is the definition of the nonatomic routing game where agents aim to minimize the conditional value-at-risk (CVaR) of the uncertain cost associated with each path. We assume that the demand is fixed and deterministic, the cost functions are continuous, and the support of the uncertainty is bounded. At an equilibrium of this game, termed CVaR-based Wardrop equilibrium (CWE), paths with nonzero flow have lowest CVaR.  Given a certain number of independent and identically distributed samples of the uncertainty, we define sample average approximation of the CVaR by replacing the expectation operator with its sample average. Subsequently, we formulate a variational inequality (VI) problem using these approximate costs. Our aim then is to study the statistical properties of the solutions of this approximate VI problem as the number of samples grow. In particular our contributions are twofold:
\begin{enumerate}
	\item We show that as the sample size grows, the set of solutions of the approximate VI problem converge almost surely, in a set-valued sense, to the set of CWE.
	\item Under the additional assumption that the costs are Lipschitz continuous, we establish the exponential convergence of the approximate solutions to the set of CWE. That is, given any constant, the probability that the distance of an approximate solution from the set of CWE is less than that constant approaches unity exponentially with number of samples. 
\end{enumerate}
We provide a simple simulation example illustrating these guarantees. 

\section{Notation and preliminaries}\label{sec:prelims}

Let $\real$, $\realnonnegative$, $\realpositive$, and $\naturalnumbers$ denote the set of real, real nonnegative, real positive, and natural numbers, respectively. Let $\norm{\cdot}$ denote the $2$-norm. We use $\until{N}:=\{1, \dots, N\}$ for $N \in \naturalnumbers$. For $x \in \real$, we let $[x]_+ = \max(x,0)$, $[x]_{-} = \min(x,0)$, and $\ceil{x}$ be the smallest integer greater than or equal to $x$. The cardinality of a set $\SS$ is denoted by $\abs{\SS}$. The distance of a point $x \in \real^m$ to a set $\SS \subset \real^m$ is denoted as $\dist(x,\SS) := \inf_{y \in \SS} \norm{x-y}$. The \myemphc{deviation} of a set $\AA \subset \real^m$ from $\SS$ is $\Db(\AA,\SS):= \sup_{y \in \AA} \dist(y,\SS)$.

\subsubsection{Variational inequality} 
Given a map $\map{F}{\real^n}{\real^n}$ and a closed set $\HH \subset \real^n$, the \myemphc{variational inequality} (VI) problem, denoted $\VI(F,\HH)$, involves finding $h^* \in \HH$ such that $(h-h^*)^\top F(h^*) \ge 0$ for all $h \in \HH$. Such a point is called a \myemphc{solution} of the VI problem. The set of solutions of $\VI(F,\HH)$ are denoted by $\SOL(F,\HH)$.

\subsubsection{Graph theory} 
A \myemphc{directed graph} is a pair $\GG = (\VV,\EE)$, where $\VV$ is a finite set called the \myemphc{vertex set} or \myemphc{node set}, $\EE \subseteq \VV \times \VV$ is called the \myemphc{edge set}, where $(i,j) \in \EE$ if there is a directed edge from vertex $i$ to $j$. A \myemphc{path} is an ordered sequence of unique vertices such that two subsequent vertices form an edge. A \myemphc{source} is a vertex with no incoming edge and a \myemphc{sink} is a vertex with no outgoing edge. 
 
\subsubsection{Uniform convergence}
A sequence of functions $\{\map{f_N}{\XX}{\YY}\}_{N=1}^\infty$, where $\XX$ and $\YY$ are Euclidean spaces, is said to \myemphc{converge uniformly} on a set $X \subset \XX$ to $\map{f}{\XX}{\YY}$ if for any $\eps > 0$, there exists $N_\eps \in \naturalnumbers$ such that
\begin{align*}
	\sup_{x \in X} \norm{f_N(x) - f(x)} \le \eps, \, \text{ for all } \, N \ge N_\eps.
\end{align*}
Similar definition applies for convergence in probability. That is, consider a random sequence of function $\{\map{f_N^\omega}{\XX}{\YY}\}_{N=1}^\infty$ defined on a probability space $(\Omega, \FF, P)$. The sequence is said to \myemphc{converge uniformly} to $\map{f}{\XX}{\YY}$ on $X$ \myemphc{almost surely} (shorthand, a.s.) if  $f_N^\omega \to f$ uniformly on $X$ for almost all $\omega \in \Omega$. 

\subsubsection{Risk measures}
Next we review notions on value-at-risk and conditional value-at-risk following~\cite{shapiro2014lectures}. Given a real-valued random variable $Z$ with probability distribution $\Pb$, we denote the \myemphc{cumulative distribution} function by $H_Z(\zeta):=\Pb(Z \le \zeta)$. The \myemphc{left-side $\alpha$-quantile} of $Z$ is defined as $H_Z^{-1}(\alpha) := \inf \setdef{\zeta}{H_Z(\zeta) \ge \alpha}$.
Given a probability level $\alpha \in (0,1)$, the \myemphc{value-at-risk} of $Z$ at level $\alpha$, denoted $\VaR_\alpha[Z]$, is the left-side $(1-\alpha)$-quantile of $Z$. Formally, 
\begin{align*}
	\VaR_\alpha[Z]  := H_Z^{-1}(1-\alpha) & = \inf \setdef{\zeta}{\Pb(Z \le \zeta) \ge 1-\alpha}
	\\
	 & = \inf \setdef{\zeta}{\Pb(Z > \zeta) \le \alpha}.
\end{align*}
The \myemphc{conditional value-at-risk (CVaR)}, also referred to as the average value-at-risk in~\cite{shapiro2014lectures}, of $Z$ at level $\alpha$, denoted $\CVaR_\alpha[Z]$, is the expectation of $Z$ when it takes values bigger than $\VaR_\alpha[Z]$. That is, 
\begin{align}\label{eq:expected-shortfall}
	\CVaR_\alpha [Z] := \Eb[Z \ge \VaR_\alpha[Z]].
\end{align}
One can show that, equivalently, 
\begin{align}\label{eq:cvar-def-alt}
	\CVaR_\alpha [Z] = \inf_{t \in \real} \Bigl\{ t + \frac{1}{\alpha} \Eb[Z - t]_+ \Bigr\}.
\end{align}
The parameter $\alpha$ characterizes the risk-averseness. When $\alpha$ is close to unity, the decision-maker is risk-neutral, whereas, $\alpha$ close to the origin implies high risk-averseness. The minimum in~\eqref{eq:cvar-def-alt} is attained at a point in the interval $[t^m,t^M]$, where $t^m : = \inf \setdef{\zeta}{H_Z(\zeta) \ge 1-\alpha}$, and $t^M := \sup \setdef{\zeta}{H_Z(\zeta) \le 1-\alpha}$.

\section{Routing game with uncertain costs} \label{sec:network-uncertain}

Consider a \myemphc{network} represented using a directed graph $\GG = (\VV, \EE)$, where $\VV$ and $\EE \subseteq \VV \times \VV$ stand for the set of nodes and edges, respectively. Here $\VV$ and $\EE$, for instance, model the sets of intersections and streets in a city when $\GG$ is a traffic network.  The sets of \myemphc{origin} and \myemphc{destination} nodes are the sets of sources and sinks in the network, and are denoted by $\OO$ and $\DD$, respectively. The set of \myemphc{origin-destination (OD) pairs} is $\WW \subseteq \OO \times \DD$. Let $\PP_w$ denote the set of available \myemphc{paths} for the OD pair $w \in \WW$ and let $\PP = \cup_{w \in \WW} \PP_w$ be the set of all paths, see Section~\ref{sec:prelims} for relevant definitions. 
We assume that numerous agents traverse through the network in a noncooperative manner. This framework is modeled as a nonatomic routing game where each individual agent's action has infinitesimal impact on the aggregate traffic flow. As a consequence, flow is modeled as a continuous variable. We assume that each agent is associated with an OD pair $w \in \WW$ and is allowed to select any path $p \in \PP_w$.  The route choices of all agents give rise to the aggregate traffic which is modeled as a \myemphc{flow vector} $h \in \realnonnegative^{\abs{\PP}}$ with $h_p$ being the \myemphc{flow} on a path $p \in \PP$. The flow between each OD pair must satisfy the travel demand. We denote the \myemphc{demand} for the OD pair $w \in \WW$ by $d_w \in \realnonnegative$ and the set of \myemphc{feasible flows} by 
\begin{equation*}
\HH := \Bigl\{h \in \realnonnegative^{\abs{\PP}} \Big| \sum_{p \in \PP_w} h_p = d_w \text{ for all } w \in \WW \Bigr\}.
\end{equation*}
Agents who choose path $p \in \PP$ experience a nonnegative uncertain \myemphc{cost} denoted as $\map{C_p}{\realnonnegative^{\abs{\PP}} \times \real^m}{\realnonnegative}$, $(h,u) \mapsto C_p(h,u)$, where $u \in \real^m$ models the \myemphc{uncertainty}. That is, the cost on a given path depends on the flow on all paths and also on a random variable. To be more precise about the uncertainty, let $(\Omega,\FF,P)$ be a probability space and $u$ be a random vector mapping into $(\real^m,B_\sigma(\real^m))$, where $B_\sigma(\real^m)$ is the Borel $\sigma$-algebra on $\real^m$. Let $\Pb$ and $\UU \subset \real^m$ be the distribution and support of $u$, respectively. We assume that $\UU$ is compact. For the cost function, we assume that for every $p \in \PP$ and $u \in \UU$, the function $h \mapsto C_p(h,u)$ is continuous. In addition, for every $p \in \PP$ and $h \in \HH$, the function $u \mapsto C_p(h,u)$ is measurable with respect to $B_\sigma(\real^m)$ and for a fixed $h \in \HH$, either $\Eb_{\Pb}[C_p(h,u)]_+$ or $\Eb_{\Pb}[C_p(h,u)]_{-}$ is finite. Here, $\Eb_{\Pb}[\, \cdot \,]$ denotes the expectation under $\Pb$ and $[\, \cdot \,]_+$ and $[\, \cdot \, ]_{-}$ denote the positive and negative parts, respectively. Additional assumptions on the cost functions will be made wherever necessary. Collecting the above described elements, a \myemphc{routing game with uncertain costs} is defined by the tuple $(\GG, \WW, \PP, C, d, \UU, \Pb)$. Note that since the cost is uncertain, one needs to assign an appropriate objective for agents which in turn defines a notion of equilibrium. In this work, we assume that agents are risk-averse and look for paths with least conditional value-at-risk (cf. Section~\ref{sec:prelims}). We assume that all agents have the same risk-aversion charecterized by the parameter $\alpha \in (0,1)$. This assumption eases notational burden and our results do hold for the general case with heterogeneous risk-aversion. Using the form~\eqref{eq:cvar-def-alt}, the $\CVaR$ associated to path $p$ as a function of the flow is
\begin{equation}\label{eq:cvar-true-cp}
\CVaR_\alpha[C_p(h,u)] = \inf_{t \in \real} \Bigl\{ t + \frac{1}{\alpha} \Eb_\Pb \left[ C_p(h,u)-t \right]_+ \Bigr\}.
\end{equation}
The notion of equilibrium then is that of Wardrop~\cite{wardrop1952some}, where the cost associated to a path is its $\CVaR$. 
\begin{definition}\longthmtitle{Conditional value-at-risk based Wardrop equilibrium (CWE)}\label{def:cwe}
	A flow vector $h^* \in \realnonnegative^{\abs{\PP}}$ is called a \myemphc{CVaR-based Wardrop equilibrium (CWE)} for the routing game with uncertain costs $(\GG, \WW, \PP, C, d, \UU, \Pb)$ if: (i) $h^*$ satisfies the demand for all OD pairs and (ii) for any OD pair $w$, a path $p \in \PP_w$ has nonzero flow if the CVaR of path $p$ is minimum among all paths in $\PP_w$. Formally, $h^*$ is a CWE if $h^* \in \HH$ and $h_p^* > 0$ for $p \in \PP_w$ only if
	\begin{align}\label{eq:CVaR-eq-regret}
	\CVaR_\alpha [C_p(h^*,u)] \le \CVaR_\alpha [C_q(h^*,u)], \quad \forall q \in \PP_w.
	\end{align}
	We denote the set of CWE by $\SSCWE \subset \HH$.
	\oprocend
\end{definition}
One can verify that the set $\SSCWE$ is equivalent to the set of solutions to the variational inequality (VI) problem $\VI(F,\HH)$ (see Section~\ref{sec:prelims} for relevant notions)~\cite{MJS:79}, where 
\begin{align*}
F_p(h) := \CVaR_\alpha [C_p(h,u)],
\end{align*}
for all $p \in \PP$. Note that the set $\HH$ is compact and convex. Further, the map $h \mapsto F(h)$ is continuous since $C_p$, $p \in \PP$ are so~\cite[Theorem 2]{rockafellar2007coherent}. Therefore, the set of solutions $\SOL(F,\HH)$ is nonempty and compact~\cite[Corollary 2.2.5]{pangVIbook-1}.  Consequently, the set $\SSCWE$ is nonempty and compact. 

The set of CWE predict flow patterns when costs are uncertain and agents behave in a risk-averse way, in particular, they minimize $\CVaR$. To compute this set, one requires to know the probability distribution $\Pb$ of the uncertainty along with the cost functionals and the fixed demand. In real-life, $\Pb$ is unknown and the decision-maker has only access to samples of the uncertainty. The objective of this paper is to provide a method to approximate the set of CWE using available samples. To this end, we define the sample average based (deterministic) approximate VI problem that acts as a surrogate to the VI problem defining the CWE. We will then study statistical properties, that is, consistency and exponential convergence, of the solutions of this approximate VI problem. Note that solving the deterministic approximate VI problem efficiently is a valid research question on its own and is not considered in the scope of this paper.

\section{Sample average approximation of CWE}

The approach in the sample average framework is to replace the expectation operator in any problem with the average over the obtained samples~\cite{shapiro2014lectures}. This is one of the main Monte Carlo methods for problems with expectations; see~\cite{demello2014survey} for a detailed survey of other sample-based techniques. In our setup, for each path of the network, we will replace the expectation operator in the definition of the $\CVaR$ of each path~\eqref{eq:cvar-true-cp} with the sample average. The thus formed set of functions result into a $\VI$ problem that approximates $\VI(F,\HH)$. 

Let $\UUhat^N := \{\uhat^1, \uhat^2, \dots, \uhat^N\}$ be the set of $N \in \naturalnumbers$ independent and identically distributed samples of the uncertainty $u$ drawn from $\Pb$. Then, the sample average approximation of the $\CVaR$ associated to path $p \in \PP$ is  
\begin{equation}\label{eq:cvar-N-cp}
\CVaRhat^N_\alpha [C_p(h,u)] := \inf_{t \in \real} \Bigl\{t + \frac{1}{N \alpha} \sum_{i=1}^N [C_p (h,\uhat^i) - t]_+ \Bigr\}.
\end{equation}
The above expression is also known as the empirical estimate of the $\CVaR$, or empirical $\CVaR$ in short. Note that the operator $\CVaRhat^N_\alpha$ is random as it depends on the realization $\UUhat^N$ of the uncertainty. Different set of samples form different set of functionals. To emphasize this dependency on the uncertainty, we represent with $\widehat{\, \cdot \,}^N$ entities that are random, dependent on the obtained samples. Using~\eqref{eq:cvar-N-cp} as the approximate cost, define the (sample-dependent) approximate variational inequality problem as $\VI(\Fhat^N,\HH)$, where 
\begin{equation*}
\Fhat^N_p (h) := \CVaRhat^N_\alpha [C_p(h,u)],
\end{equation*}
for all $p \in \PP$. We denote the set of solutions of $\VI(\Fhat^N,\HH)$ by $\SSCWEhat^N \subset \HH$. This serves as a reminder that it approximates $\SSCWE$. The notion of approximation is made precise next.

\begin{definition}\longthmtitle{Asymptotic consistency and exponential convergence}
	The set $\SSCWEhat^N$ is an asymptotically consistent approximation of $\SSCWE$, or in short, $\SSCWEhat^N$ is asymptotically consistent, if any sequence of solutions $\{\hhat^N \in \SSCWEhat^N\}_{N=1}^\infty$ has almost surely (a.s.) all accumulation points in  $\SSCWE$.  The set $\SSCWEhat^N$ is said to converge exponentially to $\SSCWE$ if for any $\eps > 0$, there exist positive constants $c_\eps$ and $\delta_\eps$ such that for any sequence $\{ \hhat^N \in \SSCWEhat^N\}_{N=1}^\infty$, the following holds
	\begin{align}\label{eq:exp-bound-conv}
	\Pb^N \Bigl(\dist(\hhat^N, \SSCWE) \le \eps \Bigr) \ge 1-c_\eps e^{-\delta_\eps N}
	\end{align}
	for all $N \in \naturalnumbers$. 
	\oprocend
\end{definition}
The asymptotic consistency of $\SSCWEhat^N$ is equivalent to saying $\Db(\SSCWEhat^N,\SSCWE) \to 0$ a.s. as $N \to \infty$. The expression~\eqref{eq:exp-bound-conv} gives a precise rate for this convergence. In our work, all convergence results are for $N \to \infty$ and so we drop restating this fact for convenience's sake.
In the following sections, we will establish the asymptotic consistency and the exponential convergence of $\SSCWEhat^N$ under suitable assumptions.
\begin{remark}\longthmtitle{Existing sample average approximations to $\CVaR$ and stochastic VI}\label{re:comparison}
	{\rm The works~\cite{meng2010saacvar} and~\cite{sun2014saa} study stochastic optimization problems where $\CVaR$ is either being minimized or used to define the constraints. Both employ the sample average approximation as proposed in~\eqref{eq:cvar-N-cp} and study asymptotic consistency and exponential convergence of the Karush-Kuhn-Tucker (KKT) points. Since $\CVaR$ is used to define a VI problem in our case, the analysis does not follow directly from these existing results. Moreover, the exponential bounds derived here are explicit, without involving ambiguous constants, than the general large deviation bounds provided in~\cite{meng2010saacvar} and~\cite{sun2014saa}. In another data-based approach~\cite{ramponi2018expectedshortfall}, the $\CVaR$ is perceived as the expected shortfall~\eqref{eq:expected-shortfall} and desirable statistical guarantees are obtained for the optimizers of its sample average. 
	}
	\oprocend
\end{remark}

\subsection{Asymptotic consistency of $\SSCWEhat^N$}\label{subsec:consistency}
We begin with stating the bound on the optimizers of the problem defining the $\CVaR$~\eqref{eq:cvar-true-cp} and the empirical $\CVaR$~\eqref{eq:cvar-N-cp}. This restricts our attention to compact domains for variables $(h,t,u)$, a property useful in showing consistency. Denote for each $p \in \PP$, functions
\begin{subequations}\label{eq:psi-maps}
	\begin{align}
	\psi_p(h,t) & := t + \frac{1}{\alpha} \Eb_\Pb [C_p(h,u)-t]_+,
	\\
	\psihat^N_p(h,t) & := t + \frac{1}{N \alpha}  \sum_{i=1}^N [C_p(h,\uhat^i) - t]_+.
	\end{align}
\end{subequations}
The map $\psihat^N_p$ is the sample average of $\psi_p$. Given our assumption that the expected value of the cost $C_p$ is bounded for any $h \in \HH$, one can deduce by strong law of large numbers~\cite{durrett2010book} that for any fixed $(h,t) \in \HH \times \real$, almost surely, $\psihat^N_p(h,t) \to \psi_p(h,t)$. We however require uniform convergence of these maps to conclude consistency, which will be established in Theorem~\ref{th:asymptotic} below. Observe that, by definition, $\CVaR_\alpha [C_p(h,u)] = \inf_{t \in \real} \psi_p(h,t)$ and $\CVaRhat^N_\alpha[C_p(h,u)] = \inf_{t \in \real} \psihat^N_p(h,t)$. The following result gives explicit bounds on the optimizers of these problems.  
\begin{lemma}\longthmtitle{Bounds on optimizers of problems defining (empirical) $\CVaR$}\label{le:cvar-opt-compact}
	For any $h \in \HH$ and $p \in \PP$, the optimizers of the problems in~\eqref{eq:cvar-true-cp} and~\eqref{eq:cvar-N-cp} exist and belong to the compact set $\TT = [m,M]$, where
	\begin{align*}
	m & := \min \setdef{C_p(h,u)}{h \in \HH, u \in \UU, p \in \PP},
	\\
	M & := \max \setdef{C_p(h,u)}{h \in \HH, u \in \UU, p \in \PP}.
	\end{align*}
	Furthermore, the set of functions 
	\begin{align}\label{eq:phi-def}
	\phi_p (h,t,u) := t + \frac{1}{\alpha} [C_p(h,u)-t]_{+},
	\end{align}
	for $p \in \PP$, satisfy for all $(h,t,u) \in \HH \times \TT \times \UU$, 
	\begin{align}\label{eq:phi-bound}
	\phi_p(h,t,u) \in \Bigl[ m, m + \frac{M-m}{\alpha} \Bigr].
	\end{align}
\end{lemma}
\begin{proof} 
	From~\cite[Chapter 6]{shapiro2014lectures}, optimizers of these problems exist and they lie in the closed interval defined by the left- and the right-side $(1-\alpha)$-quantile (cf. Section~\ref{sec:prelims}) of the respective random variables. Since this interval belongs to the set of values the functions take, we conclude that the optimizers belong to $\TT$. To conclude~\eqref{eq:phi-bound}, note that 
	\begin{align*}
	\phi_p (h,t,u) & = t + \frac{1}{\alpha} [C_p(h,u) -t]_+ \le t + \frac{1}{\alpha} [M-t]_+
	\\
	& = t + \frac{1}{\alpha} (M-t) = (1-\frac{1}{\alpha}) t + \frac{1}{\alpha} M
	\\
	& \le (1-\frac{1}{\alpha}) m + \frac{1}{\alpha} M.
	\end{align*}
	Here, the first inequality follows from the bound on $C_p$, the first equality is because $t \in [m,M]$, and the second inequality is due to the fact that $\alpha < 1$. Similarly, for the lower bound,
	\begin{align*}
	\phi_p(h,t,u) & \ge t + \frac{1}{\alpha}[m-t]_+ = t \ge m.
	\end{align*}
	This completes the proof.	
\end{proof}
We make a note here that optimizers of problems defining the $\CVaR$ in~\eqref{eq:cvar-true-cp} and~\eqref{eq:cvar-N-cp} exist and are bounded for more general cases, even when the support of the uncertainty is unbounded, see e.g.,~\cite[Chapter 6]{shapiro2014lectures}. Nevertheless, the above result provides an explicit bound which is used later in deriving precise exponential convergence guarantees. 

As a consequence of Lemma~\ref{le:cvar-opt-compact}, one can show uniform convergence of $\psihat^N_p$ to $\psi_p$. Our next step is to analyze the sensitivity of $F$ as one perturbs the underlying map $\psi$. In combination with the uniform convergence of $\psihat^N_p$, this result leads to the uniform convergence of $\Fhat^N$ to $F$. 

\begin{lemma}\longthmtitle{Sensitivity of $F$ with respect to $\psi$}\label{le:sensitivity-F}
	For any $\eps > 0$, if $\sup_{p \in \PP, (h,t) \in \HH \times \TT} \abs{\psihat^N_p (h,t) - \psi_p(h,t)} \le \eps$, where $\TT$ is defined in Lemma~\ref{le:cvar-opt-compact}, then 
	\begin{align*}
	\sup_{h \in \HH} \norm{\Fhat^N(h) - F(h)} \le \sqrt{\abs{\PP}} \eps. 
	\end{align*}
\end{lemma}
\begin{proof}
	The first step is to show the sensitivity of the map $\CVaR_\alpha [C_p(\cdot,u)]$ with respect to $\psi_p$. To this end, fix $p \in \PP$ and $h \in \HH$, and let 
	\begin{align*}
	\teehat^N_p(h) \in \underset{t \in \real}{\argmin} \, \psihat^N_p(h,t)
	\, \text{ and } \,
	\tee_p(h) \in \underset{t \in \real}{\argmin} \, \psi_p(h,t).
	\end{align*}
	These optimizers exist due to Lemma~\ref{le:cvar-opt-compact}. We now have 
	\begin{align*}
	\psi_p \Bigl(h, \tee_p(h)\Bigr) - \eps \le \psi_p \Bigl(h, \teehat^N_p(h) \Bigr) -\eps \le \psihat^N_p \Bigl(h, \teehat^N_p(h)\Bigr). 
	\end{align*}
	The first inequality is due to optimality and the second inequality holds by assumption. Similarly, one can show that 
	\begin{align*}
	\psihat^N_p\Bigl(h, \teehat^N_p(h)\Bigr) -\eps \le \psi_p\Bigl(h, \tee_p(h)\Bigr).
	\end{align*}
	The above two sets of inequalities along with the fact that $\CVaRhat^N_\alpha[C_p(h,u)] = \psihat^N_p\Bigl(h, \teehat^N_p(h)\Bigr)$ and $\CVaR_\alpha[C_p(h,u)] = \psi_p\Bigl(h, \tee_p(h)\Bigr)$ lead to the conclusion
	\begin{align}\label{eq:cvar-sup-bound}
	\sup_{h \in \HH} \Big| \CVaRhat^N_\alpha[C_p(h,u)] - \CVaR_\alpha [C_p(h,u)] \Big| \le \eps.
	\end{align} 
	Finally, the conclusion follows from the inequality $\norm{\Fhat^N(h) - F(h)} \le \sqrt{\abs{\PP}} \sup_{p \in \PP} \abs{\Fhat^N_p(h) - F_p(h)}$. 
\end{proof}
The final preliminary result states proximity of $\SSCWEhat^N$ to $\SSCWE$ given that the difference between $\Fhat^N$ and $F$ is bounded. The proof is a consequence of~\cite[Lemma 2.1]{xu2010saastochasticvi} that studies sensitivity of generalized equations and their solution sets.
\begin{lemma}\longthmtitle{Sensitivity of $\SSCWE$ with respect to $F$}\label{le:cont-sol-set}
	For any $\eps > 0$, there exists $\delta(\eps) > 0$ such that $\Db(\SSCWEhat^N, \SSCWE) \le \eps$ whenever 
	$\sup_{h \in \HH} \norm{\Fhat^N(h) - F(h)} \le \delta(\epsilon)$.
\end{lemma}
Next is the main result of this section, establishing the asymptotic consistency of $\SSCWEhat^N$. The proof puts to use the preliminary lemmas on sensitivity presented above along with the uniform convergence of $\psihat^N_p$ to $\psi_p$.

\begin{theorem}\longthmtitle{Asymptotic consistency of $\SSCWEhat^N$}\label{th:asymptotic}
	We have $\Db(\SSCWEhat^N,\SSCWE) \to 0$ almost surely.
\end{theorem}
\ifinclude{
\begin{proof}
	Consider first the a.s. uniform convergence $\psihat^N_p \to \psi_p$ over the compact set $\HH \times \TT$.  
	Note that $\psi_p(h,t) = \Eb_\Pb [\phi_p(h,t,u)]$ where $\phi_p$ is given in~\eqref{eq:phi-def} and so, $\psihat^N_p$ is the sample average of $\psi_p$. For any fixed $u \in \UU$, the map $\phi_p(\cdot, \cdot, u)$ is continuous and for any $(h,t) \in \HH \times \TT$, due to Lemma~\ref{le:cvar-opt-compact},  the map $\phi_p(h,t,\cdot)$ is dominated by the integrable function (a constant in this case) $m + \frac{M-m}{\alpha}$. Hence, by the uniform law of large numbers result~\cite[Theorem 7.48]{shapiro2014lectures}, we conclude that $\psihat^N_p \to \psi_p$ uniformly a.s. on $\HH \times \TT$. Using this fact in the sensitivity result of  Lemma~\ref{le:sensitivity-F} implies that $\Fhat^N \to F$ uniformly a.s. on the set $\HH$. Finally, we arrive at the conclusion using Lemma~\ref{le:cont-sol-set}.  
\end{proof}
}

\subsection{Exponential convergence of $\SSCWEhat^N$}\label{subsec:exp-conv}

In this section, our strategy will be to use the concentration inequality for the empirical $\CVaR$ given in~\cite{YW-FG:10-orl} and derive the uniform exponential convergence of $\Fhat^N$ to $F$. Later, we will use Lemma~\ref{le:cont-sol-set} to infer exponential convergence of $\SSCWEhat^N$. Note that the inequality given in~\cite{YW-FG:10-orl} requires compact support of the random variable and it is tight when it comes to the dependency on the risk parameter $\alpha$. For unbounded support, one can use deviation inequalities from~\cite{RKK-PLA-SPB-KJ:10-orl}.  

For a fixed $p \in \PP$ and $h \in \HH$, the deviation between the $\CVaR$ and its empirical counterpart can be bounded using the results in~\cite{YW-FG:10-orl} as
\begin{align}\label{eq:pw-CVaR-bound}
\Pb^N \Bigl(\absb{\CVaRhat_\alpha^N  & [C_p(h,u)]  - \CVaR_\alpha[C_p(h,u)]} \ge \eps \Bigr) \notag
\\
& \quad \le 6 \exp \Bigl(-\frac{\alpha \eps^2}{11 (M-m)^2} N \Bigr).
\end{align} 
In the above bound, the denominator in the exponent uses the fact that for any path and flow vector, the cost seen as a random variable is supported on the compact set $[m,M]$. Similar to the narrative of the previous section, while the above inequality holds pointwise, what we need is uniform exponential bound for proximity of $F$ to $\Fhat^N$. In the sequel, we will derive such a bound under the following condition. 
\begin{assumption}\longthmtitle{Lipschitz continuity of $C_p$}\label{as:exp-conv}
	There exists a constant $L > 0$ such that 
	\begin{align}\label{eq:lips-psi}
	\abs{C_p(h,u) - C_p(h',u)} \le L \norm{h - h'},
	\end{align}
	for all $h, h' \in \HH$, $u \in \UU$, and $p \in \PP$.
	\oprocend
\end{assumption}
Under the above Lipschitz condition on the cost functions, one can show the following. 
\begin{lemma}\longthmtitle{Lipschitz continuity of (empirical) $\CVaR$}\label{le:CVaR-Lip}
	Under Assumption~\ref{as:exp-conv}, for any path $p \in \PP$, the functions $h \mapsto \CVaRhat_\alpha^N[C_p(h,u)]$ and $h \mapsto \CVaR_\alpha[C_p(h,u)]$ are Lipschitz over the set $\HH$ with constant $\frac{L}{\alpha}$. 
\end{lemma}
\begin{proof}
	We will show the property for the function $h \mapsto \CVaRhat_\alpha^N[C_p(h,u)]$. The reasoning for $h \mapsto \CVaR_\alpha[C_p(h,u)]$ follows analogously. Consider any $h, h' \in \HH$. Recall from~\eqref{eq:psi-maps} that
	\begin{align}
	\absb{\CVaRhat^N_\alpha[C_p(h,u)] & - \CVaRhat^N_\alpha[C_p(h',u)]} \notag
	\\
	& = \absb{\inf_{t \in \real} \psihat^N_p(h,t) - \inf_{t \in \real} \psihat^N_p(h',t)}. \label{eq:cvar-psi}
	\end{align}
	Assumption~\ref{as:exp-conv} yields Lipschitz property for the map $\psihat^N_p$. To establish this, fix any $p \in \PP$ and $t \in \real$ and notice that
	\begin{align*}
	\absb{\psihat^N_p &(h,t)  - \psihat^N_p(h',t)}  = \absb{t + \frac{1}{N \alpha}  \sum_{i=1}^N [C_p(h,\uhat^i) - t]_+ 
		\\
		& \qquad \qquad - \Bigl( t + \frac{1}{N \alpha}  \sum_{i=1}^N [C_p(h',\uhat^i) - t]_+ \Bigr)}
	\\
	& \le \frac{1}{N \alpha} \sum_{i=1}^N \absb{[C_p(h,\uhat^i) - t]_+ - [C_p(h',\uhat^i) - t]_+} 
	\\
	& \le \frac{1}{N \alpha} \sum_{i=1}^N \absb{C_p(h,\uhat^i) - C_p(h',\uhat^i)} \le \frac{L}{\alpha} \norm{h - h'}.
	\end{align*}
	Above, the first relation is a consequence of the triangle inequality, the second inequality follows from the fact that the map $[ \, \cdot \, ]_+$ is Lipschitz with constant as unity, and the last inequality uses Lipschitz property of the costs. Now let $\tb, \tb' \in \real$ be such that $\psihat^N_p(h,\tb) = \inf_{t \in \real} \psihat^N_p(h,t)$ and $\psihat^N_p(h',\tb') = \inf_{t \in \real} \psihat^N_p(h',t)$. Existence of such an optimizer follows from the discussion in~\cite[Chapter 6]{shapiro2014lectures}. Now note the following sequence of inequalities that can be inferred from the optimality condition and the Lipschitz property of $\psihat^N_p$ shown above, 
	\begin{align}
	\inf_{t \in \real} \psihat^N_p(h,t) & = \psihat^N_p(h,\tb) \le \psihat^N_p(h,\tb') \notag
	\\
	& \le \psihat^N_p(h', \tb') + \frac{L}{\alpha} \norm{h - h'} \notag
	\\
	& = \inf_{t \in \real} \psihat^N_p(h',t) + \frac{L}{\alpha} \norm{h - h'}. \label{eq:inf-psi-ineq-1}
	\end{align}
	One can exchange $h$ with $h'$ in the above reasoning and obtain  
	\begin{align}
	\inf_{t \in \real} \psihat^N_p(h',t) \le \inf_{t \in \real} \psihat^N_p(h,t) + \frac{L}{\alpha} \norm{h - h'}. \label{eq:inf-psi-ineq-2}
	\end{align}
	Inequalities~\eqref{eq:inf-psi-ineq-1} and~\eqref{eq:inf-psi-ineq-2} imply that 
	\begin{align*}
		\absb{\inf_{t \in \real} \psihat^N_p(h,t) - \inf_{t \in \real} \psihat^N_p(h',t)} \le \frac{L}{\alpha} \norm{h - h'}.
	\end{align*}
	The proof concludes by using this fact in~\eqref{eq:cvar-psi}. 
\end{proof}

Next we state exponential convergence of $\Fhat^N$. The proof is largely inspired from the steps given in~\cite[Theorem 5.1]{shapiro2008saa-empec} and is a standard argument in these set of results. We note that the obtained bound is very crude and in practice, the achieved performance is much better.

\begin{proposition}\longthmtitle{Uniform exponential convergence of $\Fhat^N$ to $F$}\label{pr:exp-conv-F}
	Under Assumption~\ref{as:exp-conv}, for any $\eps > 0$, the following inequality holds for all $N \in \naturalnumbers$,
	\begin{align*}
	\Pb^N \Bigl( \sup_{h \in \HH} \norm{\Fhat^N(h) - F(h)} > \eps \Bigr) \le \gamma(\eps) e^{-\beta(\eps)N},
	\end{align*}
	where 
	\begin{subequations}\label{eq:g-b}
		\begin{align}
		\gamma(\eps) & :=   \frac{ 3 \abs{\PP} \ceil{\abs{\PP}/2}! }{\pi^{\abs{\PP}/2} } \Bigl( \frac{12 L \diam(\HH)}{\eps \alpha} \Bigr)^{\abs{\PP}} \label{eq:gamma} 
		\\
		\beta(\eps) & := \frac{\alpha \eps^2}{44 \abs{\PP} (M-m)^2}  \label{eq:beta}
		\end{align}
	\end{subequations}	
	Here, $\diam(\HH) = \sup_{h, h' \in \HH} \norm{h-h'}$ is the diameter of $\HH$.
\end{proposition}
\begin{proof}
	The idea of moving from the pointwise exponential bound~\eqref{eq:pw-CVaR-bound} to a uniform bound is to impose the pointwise bound jointly on a finite number of points and use the Lipschitz property (Lemma~\ref{le:CVaR-Lip}) to bound the deviation of the rest of the set from this finite set. Making precise the mathematical details, note that one can cover the set $\HH$ with 
	\begin{align*}
		K := \frac{\ceil{\abs{\PP}/2}! }{2 \pi^{\abs{\PP}/2} } \Bigl( \frac{12 L \diam(\HH)}{\eps \alpha} \Bigr)^{\abs{\PP}}
	\end{align*}
	number of points, labeled $\CC := \{\htil^1, \dots, \htil^K\}$, such that for any $h \in \HH$, there exists a point $\htil^{i(h)} \in \CC$ with 
	\begin{align}\label{eq:cover-cond}
		\frac{L}{\alpha} \norm{h-\htil^{i(h)}} \le \frac{\eps}{4}.
	\end{align}
	The existence of such a set of points is discussed further in Remark~\ref{re:cover} below and it relates to the covering numbers of sets. Combining the Lipschitz bound given in Lemma~\ref{le:CVaR-Lip} and the inequality~\eqref{eq:cover-cond}, we get for all $p \in \PP$ and $h \in \HH$,
	\begin{subequations}\label{eq:lips-compact}
	\begin{align}
		\absb{ \CVaRhat_\alpha^N[C_p(h,u)] - \CVaRhat_\alpha^N[C_p(\htil^{i(h)},u)]} \le \frac{\eps}{4},
		\\
		\absb{\CVaR_\alpha[C_p(h,u)] - \CVaR_\alpha[C_p(\htil^{i(h)},u)]} \le \frac{\eps}{4}.
	\end{align}	
	\end{subequations}
	The above inequalities control the deviation of $\CVaRhat_\alpha^N[C_p(\cdot,u)]$ and $\CVaR_\alpha[C_p(\cdot,u)]$ over the set $\HH$ from the values these functions take on the set $\CC$. The next step entails bounding the deviation of the $\CVaR$ and the empirical $\CVaR$ on the set $\CC$. Employing~\eqref{eq:pw-CVaR-bound} and the union bound, we have 
	\begin{align}
		&\Pb^N \Bigl( \sup_{p \in \PP, h \in \CC} \absb{\CVaRhat^N_\alpha[C_p(h,u)] - \CVaR_\alpha[C_p(h,u)]} \! \ge \! \frac{\eps}{2} \Bigr) \notag
		\\
		&  \! \le \sum_{p \in \PP} \! \sum_{h \in \CC} \! \Pb^N \! \Bigl( \absb{\CVaRhat^N_\alpha  [C_p(h,u)] \notag
			\! - \! \CVaR_\alpha \! [C_p(h,u)]} \! \ge \! \frac{\eps}{2} \Bigr) \notag
		\\
		&  \le 6 \abs{\PP} K  \exp \Bigl(-\frac{\alpha \eps^2}{44 (M-m)^2} N \Bigr). \label{eq:sup-c-bound}
	\end{align}
	The next set of inequalities characterize the difference between the $\CVaR$ and the empirical $\CVaR$ over the set $\HH$ using the Lipschitz property~\eqref{eq:lips-compact}. Fix $p \in \PP$ and let $h \in \HH$. Note that using~\eqref{eq:lips-compact},
	\begin{align*}
		& \abs{\CVaRhat^N_\alpha[C_p(h,u)] - \CVaR_\alpha[C_p(h,u)]} 
		\\
		& \le \abs{\CVaRhat^N_\alpha[C_p(h,u)] - \CVaRhat^N_\alpha[C_p(\htil^{i(h)},u)]} 
		\\
		&  \qquad \quad + \abs{\CVaRhat^N_\alpha[C_p(\htil^{i(h)},u)] - \CVaR_\alpha[C_p(\htil^{i(h)},u)]} 
		\\
		&  \qquad \quad + \abs{\CVaR_\alpha[C_p(\htil^{i(h)},u)] - \CVaR_\alpha[C_p(h,u)]}
		\\
		& \le \frac{\eps}{2} + \abs{\CVaRhat^N_\alpha[C_p(\htil^{i(h)},u)] - \CVaR_\alpha[C_p(\htil^{i(h)},u)]}.
	\end{align*}
	Next, the deviation between the $\CVaR$ and its empirical counterpart is bounded using~\eqref{eq:sup-c-bound} and the above characterization as
	\begin{align}
		& \Pb^N \Bigl( \sup_{p \in \PP, h \in \HH} \absb{\CVaRhat^N_\alpha[C_p(h,u)] \! - \! \CVaR_\alpha[C_p(h,u)]} \! \ge \! \eps \Bigr) \notag
		\\
		& \le \! \Pb^N \!  \Bigl( \sup_{p \in \PP, h \in \CC} \! \absb{ \CVaRhat^N_\alpha \! [C_p(h,t)] \! - \! \CVaR_\alpha \! [C_p(h,u)]} \! \ge \! \frac{\eps}{2} \Bigr) \notag
		\\
		& \le 6 \abs{\PP} K  \exp \Bigl(-\frac{\alpha \eps^2}{44 (M-m)^2} N \Bigr) \label{eq:uniform-exp-psi-bd}
	\end{align}
	The final step is to connect the above inequality to the difference between $\Fhat^N$ and $F$. From the proof of Lemma~\ref{le:sensitivity-F}, one can deduce that if $\sup_{h \in \HH} \norm{\Fhat^N(h) - F(h)} > \eps$, then  
	\begin{align*}
		\sup_{p \in \PP, h \in \HH} \absb{\CVaRhat^N_\alpha[C_p(h,u)] - \CVaR_\alpha[C_p(h,u)]} \! > \! \frac{\eps}{\sqrt{\abs{\PP}}}.
	\end{align*}
	Therefore, using~\eqref{eq:uniform-exp-psi-bd} we obtain
	\begin{align*}
	&\Pb^N (\sup_{h \in \HH} \norm{\Fhat^N(h) - F(h)} > \eps)
	\\
	&\le \Pb^N\Bigl( \sup_{p \in \PP, h \in \HH} \absb{\CVaRhat^N_\alpha [C_p(h,t)] 
		\\
		& \qquad \qquad \qquad \qquad - \CVaR_\alpha[C_p(h,u)] } > \frac{\eps}{ \sqrt{\abs{\PP}}} \Bigr) 
	\\
	&\le 6 \abs{\PP} K \exp \Bigl(-\frac{\alpha \eps^2}{44 \abs{\PP} (M-m)^2} N \Bigr).
	\end{align*}
	This concludes the proof.
\end{proof}

\begin{remark}\longthmtitle{A suitable cover for the set $\HH$}\label{re:cover}
	{\rm
		Here we compute the number of points $K$, that denotes the cardinality of some set $\{\htil^1, \dots, \htil^K \} \subset \HH$, required to cover the set $\HH$ according to the conditions in the proof of Proposition~\ref{pr:exp-conv-F}. In particular, for all $h \in \HH$, there exists a point $\htil^{i(h)}$, $i(h) \in \until{K}$ such that
		\begin{align*}
			\frac{L}{\alpha} \norm{h-\htil^{i(h)}} \le \frac{\eps}{4}.
		\end{align*}
		That is, $\norm{h - \htil^{i(h)}} \le \frac{\eps \alpha}{4 L}$. From~\cite{wulecture14}, this is possible with 
		$\Bigl( \frac{3}{(\eps \alpha/ 4 L)} \Bigr)^{\abs{\PP}} \frac{\vol(\HH)}{\vol(B)}$
		number of points, where $\vol(\HH)$ is the volume of the set $\HH$ and $\vol(B)$ is the volume of the unit norm ball in $\real^{\abs{\PP}}$. Since $\vol(\HH) \le \diam(\HH)^{\abs{\PP}}$ and 
		\begin{align*}
			\vol(B) \ge \frac{2 \pi^{\abs{\PP}/2}}{\ceil{\abs{\PP}/2}!},
		\end{align*} 
		we get the desired upper estimate on $K$. 
	}
	\oprocend
\end{remark}

The main result is given below. The proof with minor modifications is as given in~\cite[Theorem 2.1]{xu2010saastochasticvi}. It follows from the uniform exponential convergence of $\Fhat^N_p$. 

\begin{theorem}\longthmtitle{Exponential convergence of $\SSCWEhat^N$ to $\SSCWE$}\label{th:exp-conv}
	Let Assumption~\ref{as:exp-conv} hold. 
	Then, for any sequence $\{ \hhat^N \in \SSCWEhat^N\}_{N =1}^\infty$, $\eps > 0$, and $N \in \naturalnumbers$, the following inequality holds  
	\begin{align*}
	\Pb^N \bigl(\dist(\hhat^N, \SSCWE) \le \eps \bigr) \ge 1- \gamma(\delta(\eps)) e^{-\beta(\delta(\eps))N},
	\end{align*}	
	where $\gamma$ and $\beta$ are given in~\eqref{eq:g-b} and $\map{\delta}{\realpositive}{\realpositive}$ is a map such that the pair $(\eps,\delta(\eps))$ satisfies the condition of Lemma~\ref{le:cont-sol-set}. 
\end{theorem}
\begin{proof}
	Consider any $\eps > 0$. By Lemma~\ref{le:cont-sol-set}, if $\sup_{h \in \HH} \norm{\Fhat^N(h) - F(h)} \le \delta(\eps)$, then $\dist(\hhat^N, \SSCWE) \le \eps$. From Proposition~\ref{pr:exp-conv-F}, for any $\delta(\eps) > 0$, there exist $\gamma(\delta(\eps))$ and $\beta(\delta(\eps))$, given in~\eqref{eq:gamma} and~\eqref{eq:beta}, respectively, such that
	\begin{align*}
		\Pb^N \Bigl( \sup_{h \in \HH} \norm{\Fhat^N(h) - F(h)} > \delta(\eps) \Bigr) \le \gamma(\delta(\eps)) e^{-\beta(\delta(\eps)) N}
	\end{align*}
	for all $N$. The proof follows by using the above facts and the following set of inequalities 
	\begin{align*}
		\Pb^N (\dist(\hhat^N,&\SSCWE)   \! \le \! \eps)  \! \ge \Pb^N \! \Bigl( \sup_{h \in \HH} \norm{\Fhat^N(h) - F(h)} \! \le \! \delta(\eps) \Bigr) 
		\\
		& = 1 - \Pb^N \Bigl( \sup_{h \in \HH} \norm{\Fhat^N(h) - F(h)} > \delta(\eps) \Bigr). \hspace*{-1ex}\qed
	\end{align*}
	\renewcommand{\qedsymbol}{}
\end{proof}

\begin{remark}\longthmtitle{Sample guarantees for approximating $\SSCWE$ with $\SSCWEhat^N$}\label{re:sample}
	{\rm
		Theorem~\ref{th:exp-conv} implies that if one wants $\dist(\SSCWEhat^N,\SSCWE) \le \eps$ with confidence $1-\zeta$, where $\zeta \in (0,1)$ is a small positive number, then one would require at most  
		\begin{align*}
			N(\zeta,\eps) & = \frac{1}{\beta(\delta(\eps))} \log \Bigl( \frac{\gamma(\delta(\eps))}{\zeta} \Bigr)
			\\
			& =  \frac{44 \abs{\PP} (M-m)^2}{\alpha \delta(\eps)^2} \Bigl( \log \Bigl( \frac{ 3 \abs{\PP} \ceil{\abs{\PP}/2}! }{\pi^{\abs{\PP}/2} \zeta} \Bigr) 
			\\
			& \qquad \qquad + \abs{\PP} \log \Bigl( \frac{12 L \diam(\HH)}{\delta(\eps) \alpha} \Bigr) \Bigr)
		\end{align*} 
		number of samples of the uncertainty. Due to the exponential rate, a good feature of this sample guarantee is that $N$ depends on the accuracy $\zeta$ logarithmically. That is, one can obtain high confidence bounds with fewer samples. However, the sample size grows poorly with many other parameters, especially, the accuracy of the estimate $\eps$ and the number of paths. Further, note that to obtain an accurate sample guarantee, one needs to estimate $\delta(\cdot)$ which depends on the regularity of the cost functions. Improving the sample complexity for specific cost functions such as, piecewise affine, is part of our future work.
	}
	\oprocend
\end{remark}

\section{Simulation}\label{sec:sims}
Here we illustrate the method of sample average approximation for the computation of the CWE through an example. We consider a simple network with two nodes $\VV = \{A, B\}$ and five edges. The set of OD-pairs is $\{(A,B), (B,A)\}$. Three edges $\{1,2,3\}$ go from $A$ to $B$ and two $\{4,5\}$ go from $B$ to $A$. The set of edges form the available paths. The network and cost functions are adapted from \cite[Section 6.3]{xie2016robust}. The demand is $260$ from $A$ to $B$, and is $170$ from $B$ to $A$. The vector of cost functions is given by
\begin{equation*}
	C(h;u) = \begin{pmatrix}
	40h_1 + 20h_4 + 1000 + 3000u_1 \\
	60h_2 + 20h_5 + 950 \\
	80h_3 + 3000 \\
	8h_1 + 80h_4 + 1000 + 4000u_2 \\
	4h_2 + 100h_5 +1300
	\end{pmatrix}.
\end{equation*}
The uncertainty $u=(u_1,u_2)$ appears in an affine manner in the cost associated to edges $\{1, 4\}$. The support and distribution of both random variables is $[0,1]$ and uniform, respectively, and they are independent of each other.  We set $\alpha = 0.2$. This defines completely the routing game with uncertain costs. Since the uncertainty is additive in the costs, one can compute the $\CVaR$ of costs as
\begin{align*}
	&\begin{pmatrix} \CVaR_\alpha[C_1(h,u)] \\ \vdots \\ \CVaR_\alpha[C_5(h,u)] \end{pmatrix} 
	\\
	& \quad = \begin{pmatrix} 40h_1 + 20h_4 \\
	60h_2 + 20h_5  \\
	80h_3  \\
	8h_1 + 80h_4 \\
	4h_2 + 100h_5
	\end{pmatrix} 
	+ \begin{pmatrix}
	1000 + 3000 \CVaR_\alpha[u_1] \\ 950 \\ 3000 \\ 1000 + 4000\CVaR_\alpha [u_2] \\ 1300  
	\end{pmatrix}.
\end{align*}
The obtained cost functions are affine in the flows and so, the CWE is the solution of a linear complementarity problem (LCP)~\cite{xie2016robust}. Solving the LCP, which in this case is a convex optimization problem with quadratic cost and affine constraint, yields the unique CWE as $h^* = (89.52, 98.39, 72.09, 74.32, 95.68)$.

For the sample average approximation, we consider three scenarios with different number of samples, $N \in \{50, 500, 5000\}$. For each of these scenarios, we consider $500$ runs. Each run collects $N$ number of i.i.d samples of the uncertainty $u$, constructs the empirical $\CVaR$ costs, and computes the approximation of the CWE $\hhat^N$. Figure~\ref{fig:combined-cdf} illustrates our results. It plots the cumulative distribution function of the random variable $\norm{\hhat^N - h^*}$ as estimated using the $500$ runs. Note that the complete distribution moves to the left with increasing number of samples. This confirms our theoretical findings that as $N$ increases, the approximate solution $\hhat^N$ approaches the CWE almost surely.  

\begin{figure}
	\centering
	\includegraphics[width=0.75\linewidth]{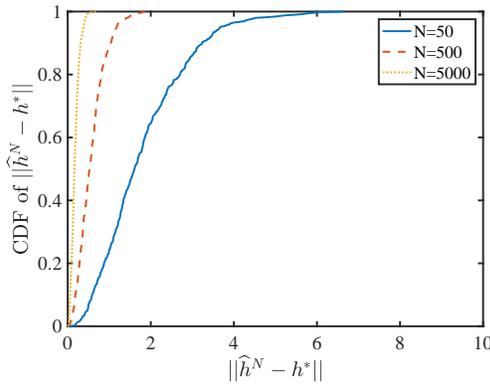}
	\caption{Plot demonstrates the convergence of the approximate solution $\hhat^N$ to the CVaR-based Wardrop equilibrium $h^*$ for the two-node five-edge example, see Section~\ref{sec:sims} for details. Each line corresponds to a different sample size and depicts the cumulative distribution of $\norm{\hhat^N-h^*}$ as obtained using $500$ runs. The lines move towards the origin as the number of samples increase depicting the convergence of $\hhat^N$ to $h^*$.}
	\label{fig:combined-cdf}
	\vspace*{-2ex}
\end{figure}

\section{Conclusions}
We have considered a nonatomic routing game with uncertain costs and defined the Wardrop equilibrium where agents opt for paths with least conditional value-at-risk. Given i.i.d samples of the uncertainty, we have investigated the statistical properties of the sample average approximation of the $\CVaR$-based Wardrop equilibrium. In particular, we have established the asymptotic consistency and the exponential convergence of the approximation scheme under suitable regularity conditions on the cost functions. Future work will involve exploring monotonicity of the deterministic VI problem formed using sample averages and designing efficient algorithms for solving it. We also wish to investigate other data-driven approaches, such as stochastic approximation routines, for computing the equilibrium. Finally, we plan to characterize the price of risk-aversion and the benefit, if any, of having heterogeneous risk-averseness of agents.

\end{document}